\documentclass[11pt]{amsart}
\usepackage{bbm}
\usepackage{amsfonts}
\usepackage{mathrsfs}
\usepackage{hyperref}
\usepackage{amssymb}
\usepackage{CJK}
  \newtheorem{thm}{Theorem}[section]
 \newtheorem{cor}[thm]{Corollary}
  \newtheorem{con}[thm]{Conjecture}
 \newtheorem{prop}[thm]{Corollary}
 \newtheorem{lem}[thm]{Lemma}
 \newtheorem{ex}[thm]{Example}
  
 \theoremstyle{remark}
 
 \numberwithin{equation}{section}
 \newcommand{\sign}{\textup{sign}}
\newcommand{\order}{\textup{o}}

\begin{document}
\title[\hfil On the Borwein conjecture]{On the Borwein conjecture}

\author{Jiyou Li}
\address{School of Mathematical Sciences, Shanghai Jiao Tong University, Shanghai, P.R. China}
\email{lijiyou@sjtu.edu.cn}

\address{Department of Mathematics, MIT, MA, USA}
\email{jiyouli@mit.edu}

\thanks{This work is supported by the National Science Foundation of China (11771280) and the National Science Foundation of Shanghai Municipal (17ZR1415400, 19ZR1424100). }
\begin{abstract}
 A conjecture of Borwein asserts that for any positive integers $n$ and $k$,  the coefficient $a_{3k}$ of $q^{3k}$ in the expansion of  $\prod_{j=0}^n (1-q^{3j+1})(1-q^{3j+2})$ is nonnegative.
 In this paper we prove that for any $0 \leq k\leq n$, there is a constant $0<c<1$ such that
 $$a_{3k}+a_{3(n+1)+3k}+\cdots+a_{3n(n+1)+3k}=\frac {2\cdot 3^{n}} {n+1}(1+O(c^n)).$$
 In particular,
 $$a_{3k}+a_{3(n+1)+3k}+\cdots+a_{3n(n+1)+3k}>0.$$
 \end{abstract}

\maketitle \numberwithin{equation}{section}
\newtheorem{theorem}{Theorem}[section]
\newtheorem{lemma}[theorem]{Lemma}
\newtheorem{example}[theorem]{Example}
\allowdisplaybreaks

\section{Introduction}
In this note we investigate the following problem raised by Borwein.
\begin{con} [P. Borwein, 1990]
For the coefficients $a_j$ defined by $$\prod_{j=0}^n (1-q^{3j+1})(1-q^{3j+2})=\sum_{j\geq 0}a_j q^j,$$
if $j\equiv0\pmod{3}$, then $a_j\geq 0$, and else $a_j\leq 0$.
\end{con}

The conjecture is actually the first one among three conjectures raised by Borwein.  Another two ask similar positivity questions for $\prod_{j=0}^n ((1-q^{3j+1})(1-q^{3j+2}))^2$  and $\prod_{j=0}^n (1-q^{5j+1})(1-q^{5j+2})(1-q^{5j+3})(1-q^{5j+4})$.
Recently, a remarkable proof of the first conjecture was given by Wang \cite{Wang} using analytic methods.  Unfortunately, the remaining conjectures are still open. More about the story and more related work can be found in \cite{A, ABB, B, RP, W1, W2, Z}. The purpose of this note is to investigate these conjectures from an enumerative approach and to find potential relations to number theory.

If we define three polynomials $A_n(x), B_n(x)$, and $C_n(x) $ by
$$\prod_{j=1}^n (1-q^{3j+1})(1-q^{3j+2})=A_n(q^3)-qB_n(q^3)-q^2C_n(q^3),$$
then the conjecture is equivalent to saying that all the three polynomials have non-negative coefficients.
One may write $A_n(q)=\sum_{i}a_{3i} q^i$
 as a sum of $q$-binomial coefficients with alternating
signs by the $q$-binomial theorem (for a proof, see \cite{A}) as follows:
$$A_n(q)=\sum_{k=-\lfloor n/3\rfloor}^{\lfloor n/3\rfloor}(-1)^k q^{k(9k-1)/2}{2n \choose n+3k }_q.$$
Here ${2n \choose n+3k }_q$ is the Gaussian binomial coefficients.
Unfortunately, it seems not obvious why the coefficients in this sum
 are always non-negative.

On the other hand, it is clear the number $$a_j=C_e(j, n)-C_o(j, n),$$ where
$C_e(j, n)$ (respectively $C_o(j, n)$) is the number of partitions of $j$ into an even (respectively an odd) number of distinct non-multiples of 3 each less than $3n+3$.

In this note we will take the advantage of regarding
$C_e(j, n)$ (respectively $C_o(j, n)$) as counting suitable subsets in
a particular subset. Precisely, let $D=\{1, 2, 4, 5, \cdots, 3n+1, 3n+2\}$.
Then    \begin{align*}
C_e(j, n)&=\#\{S\subseteq D, |S|\equiv 0 \pmod{2}, \sum_{x\in S}x=j\};\\
C_o(j, n)&=\#\{S\subseteq D, |S|\equiv 1 \pmod{2}, \sum_{x\in S}x=j\}.
 \end{align*}
Then the conjecture is equivalent to a typical \texttt{"}parity counting\texttt{"} problem in subsets of integers, which is to determine if
$C_e(j, n)\geq C_o(j, n).$

We fix a positive integer $n$ and assume $0\leq j\leq 3(n+1)^2$.
In this note we shall focus on the study of $a_j=C_e(j, n)-C_o(j, n)$.
It turns out to be a challenging task to give a reasonable estimate for $a_j$, though the following has been shown when $n$ is replaced by $\infty$ \cite{A}.
\begin{thm}[Garvan, Borwein and Andrews]
 For any prime $p$, let
  $$\frac{\prod_{j=1}^\infty (1-q^j)}{ \prod_{j=1}^\infty (1-q^{pj})}=\sum_{j\geq 0}b_{p, j} q^j.$$
  Then for any $j\geq 0$, $$b_{p, j} b_{p, j+p}\geq 0.$$
\end{thm}

For the case $p\mid j$, Stanley  \cite{St} gives a direct proof of this result. His proof also expresses $b_{p, j}$ as a sum of two restricted partition functions (here \text{``}restricted\text{"} means each part of the partition does not belong to some congruence classes) of $j/p$.
 \begin{thm}[Stanley]\label{thm1.3}
Suppose  $b_{p, j}$ are defined as above.
   Then for any $p\mid j\geq 0$, $$b_{p, j} b_{p, j+p}\geq 0.$$
    Moreover,
     \begin{align*}
     b_{p, pj}=P_{\not\equiv 0,  \frac {3p\pm1}{2}\pmod{3p}}(j)+P_{\not\equiv 0, \frac {(3-2t)p-1}{2}, \frac {(3+2t)p+1}{2}\pmod{3p}}(j-{\frac {t(pt+1)}6}),
      \end{align*}
    where $P$ counts the number of restricted partitions (for instance,
    the first term in the right side of the equality
     counts the number of partitions of $j$ whose parts do not belong to $0, \frac {3p-1}{2}, \frac {3p+1}{2}\pmod{3p}$ and $t$ is the smallest positive integer such that $3 \mid (pt+1)$ (if such $t$ does not exist, then the corresponding $P$ is 0).
\end{thm}

For interested readers we include the proof of Theorem \ref{thm1.3} in Section 4.

\begin{cor}
Let $P^*_{e, \not\equiv 0 \pmod{3}}(3j)$ (respectively $P^*_{o, \not\equiv 0\pmod{3}}(3j)$ be the number of partitions of $3j$ into an even (respectively an odd) number of distinct non-multiples of 3.
Then we have the equality
$$ P_{\not\equiv 0, 4, 5\pmod{9}}(j)=P^*_{e, \not\equiv 0\pmod{3}}(3j)-P^*_{o, \not\equiv 0\pmod{3}}(3j).$$

\end{cor}

\begin{proof}
By the counting formula in Theorem \ref{thm1.3} for the case $p=3$, $b_{3, 3j}$ is exactly the numbers $a_{3j}$ appearing in the Borwein conjecture  for $n=\infty$. Explicitly
 if we let
 $$\prod_{j=0}^{\infty} (1-q^{3j+1})(1-q^{3j+2})=\sum_{j\geq 0}b_{3, 3j} q^j,$$
 then $$b_{3, 3j}=P_{\not\equiv 0, 4, 5\pmod{9}}(j).$$

 Recall $a_j=C_e(j, n)-C_o(j, n)$ and  for $n=\infty$ we have $$b_{3, 3j} =P^*_{e, \not\equiv 0\pmod{3}}(3j)-P^*_{o, \not\equiv 0\pmod{3}}(3j).$$
  The equality then follows from Theorem \ref{thm1.3}.
\end{proof}

 For instance, when $j=24$, $P^*_{e, \not\equiv 0\pmod{3}}(3j)=1597$
 $P^*_{o, \not\equiv 0\pmod{3}}(3j)=1117$ and $P_{\not\equiv 0, 4, 5\pmod{9}}(j)=480$. When $j=33$, $P^*_{e, \not\equiv 0\pmod{3}}(3j)=10509$
 $P^*_{o, \not\equiv 0\pmod{3}}(3j)=8283$ and $P_{\not\equiv 0, 4, 5\pmod{9}}(j)=2226$.
   This observation indicates the hardness of the sign decision of $a_j=C_e(j, n)-C_o(j, n)$ in some sense.

For finite $n$, all the above methods do not apply because of the lack of modularity. Instead we give some interesting explicit counting formulae and
thus obtain some partial results on Borwein's Conjecture.
The main observation is to view the evaluation of $C_e(j, n)$ and $C_o(j, n)$ as the problem of counting subsets in a special subset of $\mathbb{Z}_{3n+3}$.

\begin{thm}\label{thm1.2}
Let $N=3n+3$, $D=\{1, 2, 4, 5, \dots, 3n+1, 3n+2\}\subseteq \mathbb{Z}_N$.
Let $M(k, b, D)$ be the number of $k$-subsets whose sum is $b$, i.e.,
$$M(k, b, D)=\#\{S\subseteq D, |S|=k, \sum_{x\in S}x=b\}.$$
For convenience, the sum of an empty set is defined to be $0$.
Let $$M(b)=\sum_{2 \mid k}M(k, b, D)-\sum_{2 \nmid k}M(k, b, D).$$ If  $b \in 3\mathbb{Z}_N$, then
$$|M(b)-\frac {2\cdot 3^{N/3}} {N}|\leq 2^{\frac N3}+3^{\frac N6}.$$
  In particular, there exists a $c\in(0, 1)$ such that
\begin{align*}
M(b)=\frac {2\cdot 3^{N/3}} {N}(1+O(c^N)).
  \end{align*}

\end{thm}

\begin{prop}
 Suppose $M(b)$ is defined as above. Then $M(b)>0$.
\end{prop}

\begin{proof}
For $N \leq 18$, a direct computer check gives $M(b)>0$.
For $N \geq 21$, one checks that $M(b)\geq \frac {2\cdot 3^{N/3}} {N}-2^{\frac N3}-3^{\frac N6}>0$.
\end{proof}

\begin{ex}
Choose $n=2$, $b=3$, and thus $N=9$. Then one computes $M(b)=6$.
In $D=\{1, 2, 4, 5, 7, 8\}\subseteq \mathbb{Z}_9$, there are $7$ even subsets which sum to $3$: $3=1+2=5+7=4+8=1+2+5+7=1+2+4+8=4+5+7+8=1+2+4+5+7+8$ and $1$ odd subset: $3=1+4+7$.
\end{ex}

\begin{ex}
Choose $n=20$, $b=0$, and thus $N=60$. Then one computes $M(b)=116228143
$ and the above bound gives $M(b)\approx 116226147$.
\end{ex}

The proof is based on a subset counting technique by Li-Wan and is given in Section 3.

{\bf Remark.}  The proof can be generalized to any prime $p>3$.

 Based on Theorem \ref{thm1.2}, we obtain an almost explicit counting formula for the partial sums of $a_j's$ lying in the same class modulo an integer. This leads to the positivity of incomplete sums of some $a_{3j}$'s and hence provides evidence to the truth of the conjecture.

\begin{thm}
Let $N=3(n+1)$ and let $$\prod_{j=0}^n (1-q^{3j+1})(1-q^{3j+2})=\sum_{j\geq 0}a_j q^j.$$
For $0\leq j\leq N$ and $j\equiv0\pmod{3}$, we have
$$a_j+a_{N+j}+\cdots+a_{nN+j}=\frac {2\cdot 3^{N/3}} {N}(1+o(N)).$$
In particular, $$a_j+a_{N+j}+\cdots+a_{nN+j}>0.$$
\end{thm}

\begin{proof}
This is a direct consequence of the Theorem \ref{thm1.2}, since $M(j)=a_j+a_{N+j}+\cdots+a_{nN+j}$.
\end{proof}

\section{A distinct coordinate sieving formula}
Let $\psi$ denote an additive character from $\mathbb{Z}_N$ to the group of all nonzero complex numbers $\mathbb{C}^*$  and let  $\psi_0$ be the principal character sending each element in $\mathbb{Z}_N$ to 1. Let
$\hat{\mathbb{Z}}_N$ be the group of additive characters of $\mathbb{Z}_N$.
It is well-known that $\hat{\mathbb{Z}}_N$ is isomorphic to $\mathbb{Z}_N$.
Denote by $\order(\chi)$ the order of $\chi$.

Let $N=3n+3$, $D=\{1, 2, 4, 5, \dots, 3n+1, 3n+2\}\subseteq \mathbb{Z}_N$ and $M(k, b, D)$ be defined by
$$M(k, b, D)=\#\{S\subseteq D, |S|=k, \sum_{x\in S}x=b\}.$$
 Let $X=D^k$ and $\overline{X}$ be the vectors in $X$ with distinct coordinates. The circle method gives
\begin{align*}
&k!M(k, b, D)\\
&={N^{-1}} \sum_{(x_1, x_2,\dots x_k) \in
\overline{X}}
\sum_{\psi\in \widehat{\mathbb{Z}}_N}\psi(x_1+x_2+\cdots +x_k-b)\\
 &={N^{-1}} (2N/3)_k+N^{-1} \sum_{\psi\ne \psi_0}\sum_{(x_1,
x_2,\cdots x_k) \in\overline{X}}\psi(x_1)\psi(x_2)\cdots \psi(x_k)\psi^{-1}(b)\\
&={N^{-1}}  {(2N/3)_k}+{N^{-1}} \sum_{\psi\ne
\psi_0}\psi^{-1}(b)\sum_{(x_1,x_2,\dots x_k)
\in\overline{X}}\prod_{i=1}^{k} \psi(x_i).
\end{align*}

It is nontrivial to evaluate the sum $\overline{f}=\sum_{(x_1,x_2,\dots x_k) \in\overline{X}}\prod_{i=1}^{k} \psi(x_i)$ for non principal  characters. For simplicity, we explain our approach by investigating $|\overline{X}|$ as an example, since $\overline{f}$ can be naturally regarded as a weighted version of $\overline{X}$.

 Define $X_{ij}=\{ (x_1,x_2,\cdots,x_k)\in X, x_i=x_j \}$.
Then the classical inclusion-exclusion sieving gives
 \begin{align*}
|\overline{X}|&=|X|-|\bigcup_{1\leq i<j\leq k}{X_{ij}}|\\
&=|X|-\sum _{1\leq i<j\leq k}|X_{ij}|+\sum_{1\leq i<j\leq k,1\leq
s<t\leq k, (i,j)\neq(s,t)} |X_{ij}\bigcap X_{st}|\\
&-\cdots+(-1)^{k \choose 2} |\bigcap_{1\leq i<j\leq k}X_{ij}|.
 \end{align*}

Unfortunately, there are totally $2^{k \choose 2}$ terms in the summations and thus even nice estimates for each $|X_{ij}|$ would
sum up to a very large error when $k$ is large. In many applications, people use Bonferroni inequalities such as $|\overline{X}|\geq |X|-\sum _{1\leq i<j\leq k}|X_{ij}|$ and so on. However, one of the main disadvantage of these inequalities is that the important parameter $k$ is usually bounded by a square root of $|X|$.

 A sieving formula discovered by Li-Wan \cite{LW2} overcomes this bottle-neck in many cases. It significantly improves the above classical inclusion-exclusion sieving for many cases. We cite it here without any
proof. For details and some other related applications, please refer to
\cite{LW2,LW3}.

Let $S_k$ be the symmetric group on $k$ elements. It is well known
that every permutation $\tau\in S_k$ factorizes uniquely
as a product of disjoint cycles and  each
  fixed point is viewed as a trivial cycle of length $1$.
  For $\tau\in S_k$, define $\sign(\tau)=(-1)^{k-l(\tau)}$, where $l(\tau)$ is the number of
 cycles of $\tau$ including the trivial cycles.

Suppose $X$ is a finite set of vectors of length $k$ over an
alphabet set $D$.  Define $\overline{X}=\{(x_1,x_2,\cdots,x_k)\in X
\ | \ x_i\ne x_j, \forall i\ne j\}.$ Let $f(x_1,x_2,\dots,x_k)$ be a
complex valued function defined over $X$ and let $$F=\sum_{x \in
\overline{X}}f(x_1,x_2,\dots,x_k)\label{1.00}.$$
Many problems arising from number theory and coding theory rely on giving a good estimate for $F$.

The formula of Li and Wan reduce the problem of estimating $F$
to estimating some sums determined by a permutation in $S_k$, which is usually more convenient to compute.

 For a permutation $\tau\in S_k$,  assume we have the factorization  $$\tau=(i_1i_2\cdots i_{a_1})
  (j_1j_2\cdots j_{a_2})\cdots(l_1l_2\cdots l_{a_s})$$
  with $1\leq a_i, 1 \leq i\leq s$, and we define  $F_{\tau}=\sum_{x \in X_{\tau} } f(x_1,x_2,\dots,x_k),$ where
 \begin{align}\label{1.1}
     X_{\tau}=\left\{
(x_1,\dots,x_k)\in X,
 x_{i_1}=\cdots=x_{i_{a_1}},\dots, x_{l_1}=\cdots=x_{l_{a_s}}.
 \right\}
\end{align}

\begin{thm} \label{thm1.0}
 We have
\begin{align}
  \label{1.5} F=\sum_{\tau\in S_k}{\sign(\tau)F_{\tau}}.
    \end{align}
 \end{thm}

Note that the symmetric group $S_k$ acts on $D^k$ naturally by
permuting coordinates. That is, for $\tau\in S_k$ and
$x=(x_1,x_2,\dots,x_k)\in D^k$,  $\tau\circ
x=(x_{\tau(1)},x_{\tau(2)},\dots,x_{\tau(k)}).$
  A subset $X$ in $D^k$ is said to be symmetric if for any $x\in X$ and
any $\tau\in S_k$, $\tau\circ x \in X $. In particular,  if  $X$ is
symmetric and $f$ is a symmetric function under the action of $S_k$,
we then have the following formula which is simpler than
(\ref{1.5}).
\begin{prop} \label{thm1.1} Let $C_k$ be the set of conjugacy  classes
 of $S_k$.  If $X$ is symmetric and $f$ is symmetric, then
 \begin{align}\label{7} F=\sum_{\tau \in C_k}\sign(\tau) C(\tau)F_{\tau},
  \end{align} where $C(\tau)$ is the number of permutations conjugate to
  $\tau$.

\end{prop}

For the purpose of more explicit computation of the above summation, we need several combinatorial formulas. Recall that a permutation $\tau\in S_k$ is
said to be of type $(c_1,c_2,\cdots,c_k)$ if $\tau$ has exactly
$c_i$ cycles of length $i$ and $\sum_{i=1}^k ic_i=k$. Let
$N(c_1,c_2,\dots,c_k)$ be the number of permutations in $S_k$ of
type $(c_1,c_2,\dots,c_k)$. It is well known that (see for example \cite{St1})
$$N(c_1,c_2,\dots,c_k)=\frac {k!} {1^{c_1}c_1! 2^{c_2}c_2!\cdots k^{c_k}c_k!}.$$

\begin{lem} \label{lem2.6}
Suppose $k$ is a positive integer and $t_1, t_2, \ldots, t_k\in \mathbb{R}$ are parameters. Define the formal sum by
\begin{align*}C_k(t_1,t_2,\dots,t_k)= \sum_{\sum
ic_i=k} N(c_1,c_2,\dots,c_k)t_1^{c_1}t_2^{c_2}\cdots t_k^{c_k},
 \end{align*}
 where the summation is through all partitions of $k$.

 \item Case 1:  $t_i=a$ iff $d\mid i$ and $t_i=0$ iff $d\nmid i$. Then
  \begin{align*}
C_k(\overbrace{0, \ldots, 0}^{d-1}, a, \overbrace{0,\ldots, 0}^{d-1}, a, \ldots) &=\left[\frac {u^k}{k!}\right] \frac 1 {(1-u^d)^{\frac {a}{d}}}.
\end{align*}

 \item Case 2:  $t_i=a$ iff $d\mid i$ and $3d\nmid i$; $t_i=b$ iff $3d\mid i$; and $t_i=0$ iff $d\nmid i$. Then
  \begin{align*}
C_k(\overbrace{0, \ldots, 0}^{d-1}, a, \overbrace{0,\ldots, 0}^{d-1}, a, \overbrace{0, \ldots, 0}^{d-1}, b, \ldots) &=
\left[\frac {u^k}{k!}\right] \frac 1 {(1-u^d)^{\frac {a}{d}}}\frac 1 {(1-u^{3d})^{\frac {b-a}{3d}}}.
\end{align*}
\end{lem}

\begin{proof}
 Define the exponential generating function (see for example, \cite{LW2}) by
$$\sum_{k\geq 0}C_k(t_1,t_2,\cdots,t_k)\frac {u^k}{k!}=e^{ut_1+u^2\cdot\frac{t_2}2+
u^3\cdot\frac {t_3}3+\cdots}.$$
In the case 1, the condition is $t_i=a$ for $d\mid i$ and $t_i=0$ for $d\nmid i$.
Thus we deduce that
\begin{align*} C_k(\overbrace{0, \ldots, 0}^{d-1}, a, \overbrace{0,\ldots, 0}^{d-1}, a, \ldots)&=\left[\frac {u^k}{k!}\right]e^{u^d\cdot\frac ad+u^{2d}\cdot\frac {a} {2d}+u^{3d}\cdot\frac {a} {3d}+\cdots}\\
&=\left[\frac {u^k}{k!}\right]e^{\frac {a}{d}\sum_{i}\frac {(u^d)^i}{i}}\\
&=\left[\frac {u^k}{k!}\right]e^{-\frac {a}{d}\log{\left(1-u^d\right)}}\\
&=\left[\frac {u^k}{k!}\right] \frac 1 {(1-u^d)^{\frac {a}{d}}}.
\end{align*}

  In the case 2, the condition is $t_i=a$ for $d\mid i$ and $3d\nmid i$; $t_i=b$ for $3d\mid i$; and $t_i=0$ for $d\nmid i$.
Similarly we deduce that
\begin{align*} & C_k(\overbrace{0, \ldots, 0}^{d-1}, a, \overbrace{0,\ldots, 0}^{d-1}, a, \overbrace{0, \ldots, 0}^{d-1}, b, \ldots)\\
 &=\left[\frac {u^k}{k!}\right]e^{u^d\cdot\frac ad+u^{2d}\cdot\frac {a} {2d}+u^{3d}\cdot\frac {b} {3d}+u^{4d}\cdot\frac {a}{4d}+u^{5d}\cdot\frac {a} {5d}+u^{6d}\cdot\frac {b} {6d}\cdots}\\
&=\left[\frac {u^k}{k!}\right]e^{\frac {a}{d}\sum_{i}\frac {(u^d)^i}{i}+\frac {b-a}{3d}\sum_{i}\frac {(u^{3d})^i}{i}}\\
&=\left[\frac {u^k}{k!}\right]e^{-\frac {a}{d}\log{\left(1-u^d\right)}-\frac {b-a}{3d}\log{\left(1-u^{3d}\right)}}\\
&=\left[\frac {u^k}{k!}\right] \frac 1 {(1-u^d)^{\frac {a}{d}}}\frac 1 {(1-u^{3d})^{\frac {b-a}{3d}}}.
\end{align*}
The proof is complete.
\end{proof}
 \section{Proof of Theorem \ref{thm1.2}}
\begin{proof}
We use the notation defined in Section 2.
Recall  $N=3n+3$, $D=\{1, 2, 4, 5, \dots, 3n+1, 3n+2\}\subseteq \mathbb{Z}_N$ and $M(k, b, D)$ is defined by
$$M(k, b, D)=\#\{S\subseteq D, |S|=k, \sum_{x\in S}x=b\}.$$
 Let $X=D^k$ and $\overline{X}$ be the vectors in $X$ with distinct coordinates. In the beginning of Section 2, we have shown that
\begin{align*}
k!M(k, b, D)={N^{-1}}  {(2N/3)_k}+{N^{-1}} \sum_{\psi\ne
\psi_0}\psi^{-1}(b)\sum_{(x_1,x_2,\dots x_k)
\in\overline{X}}\prod_{i=1}^{k} \psi(x_i).
\end{align*}

In this section, we will use Li-Wan's sieving formula to give a
better estimate on $\sum_{(x_1,x_2,\dots x_k) \in\overline{X}}\prod_{i=1}^{k} \psi(x_i)$ for non principal characters.

Denote $\prod_{i=1}^{k}\psi(x_i)$ by  $f_{\psi}(x)$.  For
  $\tau\in S_k$,  let
$$F_{\tau}(\psi)=\sum_{x\in X_{\tau}}f_{\psi}(x),$$
where $X_{\tau}$ is defined as in (\ref{1.1}). Obviously $X$ is
symmetric and $f(x)=f_{\psi}(x_1,x_2,\dots,x_{k})$ is normal on $X$.
Applying (\ref{7}) in Corollary \ref{thm1.1}, we have
  \begin{align}\label{3.1}
 k!M(k, b, D)&={N^{-1}} {(2N/3)_k}+{N^{-1}} \sum_{\psi\ne \psi_0}\psi^{-1}(b) \sum_{\tau\in C_{k}}\sign(\tau)C(\tau) F_{\tau}(\psi),
  \end{align}
 where $C_{k}$ is the set of conjugacy classes
 of $S_{k}$ and $C(\tau)$ is the number of permutations conjugate to $\tau$.  Suppose $\tau$ has type $(c_1, c_2, \dots, c_k)$, then we compute that
   \begin{align*}
F_{\tau}(\psi)&=\sum_{x \in X_{\tau}}\prod_{i=1}^{k} \psi(x_i)\\
&=\sum_{x \in X_{\tau}}\prod_{i=1}^{c_1} \psi(x_i)\prod_{i=1}^{c_2}
\psi^2(x_{c_1+2i})\cdots\prod_{i=1}^{c_k} \psi^k(x_{c_1+c_2+\cdots+k i})\\
 &=\prod_{i=1}^{k}(\sum_{a\in D}\psi^i(a))^{c_i}.
 \end{align*}
 It then suffices to compute $s_{\chi}(D)=\sum_{a\in
D}\chi{(a)}$ for all characters $\chi$.

 Now let $N=3n+3$ and $D=\{1, 2, 4, 5, \cdots, 3n+1, 3n+2\}\subseteq \mathbb{Z}_N$.

If  $\order(\chi)=1$, then clearly $s_{\chi}(D)=2n+2$.

If $\order(\chi)=3$,  then $\chi$ is trivial on $\mathbb{Z}_N-D$ from the fact that $\mathbb{Z}_N-D$ is an index 3 subgroup of $\mathbb{Z}_N$.
By the equality $s_{\psi}(\mathbb{Z}_N)=0$ we have  $s_{\chi}(D)=0-(n+1)=-n-1$.

If  $\order(\chi) \nmid 3$, then $\chi$ is a nontrivial character on the group $\mathbb{Z}_N-D$. Thus $s_{\chi}(\mathbb{Z}_N-D)=0$. Since $s_{\chi}(\mathbb{Z}_N)=0$,  we have $s_{\chi}(D)=0$.

Note that $\order(\psi^i)=\frac {\order(\psi)}{(\order(\psi), i)}$.
Thus if $3 \nmid {\order(\psi)}$, then one always has $3 \nmid {\order(\psi^i)}$. Thus if $\tau$ is of type $(c_1, c_2, \ldots, c_k)$, then we have
  \begin{align*}
F_{\tau}(\psi)&=\prod_{i=1}^{k}(\sum_{a\in D}\psi^i(a))^{c_i}\\
&=\prod_{i, \psi^i=1}(\sum_{a\in D}\psi^i(a))^{c_i}\prod_{i, \psi^i\ne1}(\sum_{a\in D}\psi^i(a))^{c_i}\\
&= \prod_{i, \psi^i=1} (2N/3)^{c_i} \prod_{i, \psi^i \ne 1} 0^{c_i}.
 \end{align*}
 The case for $3 \mid {\order(\psi)}$ is a bit more complicated. Here, we use
  \begin{align*}
F_{\tau}(\psi)&=\prod_{i=1}^{k}(\sum_{a\in D}\psi^i(a))^{c_i}\\
&=\prod_{i, \psi^i=1}(\sum_{a\in D}\psi^i(a))^{c_i}\prod_{i, \order(\psi^i)=3}(\sum_{a\in D}\psi^i(a))^{c_i}\prod_{i, \order(\psi^i)\ne 1, 3}(\sum_{a\in D}\psi^i(a))^{c_i}\\
&=\prod_{i, \psi^i=1}(2N/3)^{c_i}\prod_{i, \order(\psi^i)=3}(-N/3)^{c_i}\prod_{i, \order(\psi^i)\ne 1, 3}0^{c_i}.
 \end{align*}
 Thus  for $3\mid b$, by \ref{3.1} we obtain  that
  \begin{align} \label{3.2}
 k!M(k, b, D)={N^{-1}} {(2N/3)_k}+I+II,
 \end{align}
where
$$I={N^{-1}} \sum_{\psi\ne \psi_0, 3\nmid\text{o}(\psi)}\psi^{-1}(b) \sum_{\tau\in C_{k}}\sign(\tau)C(\tau) F_{\tau}(\psi)$$ and
$$II={N^{-1}} \sum_{3\mid\order(\psi)}\psi^{-1}(b) \sum_{\tau\in C_{k}}\sign(\tau)C(\tau) F_{\tau}(\psi).$$
For simplicity we use the notation defined in Lemma \ref{lem2.6} that $$C_k(t_1,t_2,\dots,t_k)= \sum_{\sum ic_i=k} N(c_1,c_2,\dots,c_k)t_1^{c_1}t_2^{c_2}\cdots t_k^{c_k},$$
where $N(c_1,c_2,\dots,c_k)$ is the number  of permutations in $S_k$ of
type $(c_1,c_2,\dots,c_k)$.

We first compute I. Note that $\sign({\tau})=(-1)^{k-\sum_{i}c_i}$ when $\tau$ is of type $(c_1, c_2, \ldots, c_k)$. Thus we have
  \begin{align*}
  I&={N^{-1}} \sum_{\psi\ne \psi_0, 3\nmid\text{o}(\psi)}\psi^{-1}(b) \sum_{\tau\in C_{k}}\sign(\tau)C(\tau) F_{\tau}(\psi)\\
  &=(-1)^k{N^{-1}} \sum_{\psi\ne \psi_0, 3\nmid\text{o}(\psi)}\psi^{-1}(b)\sum_{\sum ic_i=k}N(c_1,c_2,\dots,c_k)\prod_{i, \psi^i=1} (-2N/3)^{c_i} \prod_{i, \psi^i \ne 1} 0^{c_i}\\
  &=(-1)^k{N^{-1}} \sum_{\psi\ne \psi_0, 3\nmid\text{o}(\psi)}\psi^{-1}(b) C_k(\overbrace{0, \ldots, 0}^{\text{o}(\psi)-1}, -2N/3, \overbrace{0,\ldots, 0}^{\text{o}(\psi)-1}, -2N/3, \ldots).
 \end{align*}
Thus by  Lemma \ref{lem2.6},
\begin{align*}
       I&=(-1)^k{N^{-1}} \sum_{3\nmid\text{o}(\psi), 1\ne \order(\psi)\mid k}\psi^{-1}(b)\left[\frac {u^k}{k!}\right]  {(1-u^{\order(\psi)})}^{{\frac {2N}{3\order(\psi)}}}.
   \end{align*}
Similarly, we compute that
\begin{align*}
&II={N^{-1}} \sum_{3\mid\order(\psi)}\psi^{-1}(b) \sum_{\tau\in C_{k}}\sign(\tau)C(\tau) F_{\tau}(\psi)\\
  &=(-1)^k{N^{-1}} \sum_{3\mid\order(\psi)}\psi^{-1}(b)\sum_{\sum ic_i=k}N(c_1,\dots,c_k)\prod_{i, \psi^i=1}(-2N/3)^{c_i}\prod_{i, \order(\psi^i)=3}(N/3)^{c_i}\prod_{i, \order(\psi^i)\ne 1, 3}0^{c_i}\\
    &=(-1)^k{N^{-1}} \sum_{3\mid\text{o}(\psi)}\psi^{-1}(b) C_k(\overbrace{0, \ldots, 0}^{\text{o}(\psi)/3-1}, N/3, \overbrace{0,\cdots, 0}^{\text{o}(\psi)/3-1}, N/3, \overbrace{0,\cdots, 0}^{\text{o}(\psi)/3-1}, -2N/3, \ldots).
   \end{align*}

Again by Lemma \ref{lem2.6} we have
\begin{align*}
       II&=(-1)^k k!{N^{-1}} \sum_{3\mid\text{o}(\psi)\mid k}\psi^{-1}(b)\left[\frac {u^k}{k!}\right] \frac {       {(1-u^{\text{o}(\psi)})^{\frac {N}{\text{o}(\psi)}}}} {(1-u^{\text{o}(\psi)/3})^{\frac {N}{\text{o}(\psi)}}}\\
       &=(-1)^k k!{N^{-1}} \sum_{3\mid\text{o}(\psi)\mid k}\psi^{-1}(b)\left[\frac {u^k}{k!}\right]
       (1+u^{\text{o}(\psi)/3}+u^{2\text{o}(\psi)/3})^{\frac {N}{\text{o}(\psi)}}.
            \end{align*}
 In the summation for II, we partition the characters into two parts: characters of order 3 and characters of larger degree. Let
III and IV be the corresponding sums. Thus for $k\leq 2N/3$,
\begin{align*}
 III&= (-1)^k  \frac  2N \left[\frac {u^k}{k!}\right]  \frac {(1-u^3)^{N/3}}{(1-u)^{N/3}}\\
 &=(-1)^k  \frac  2N \left[\frac {u^k}{k!}\right] (1+u+u^2)^{N/3}.
 \end{align*}
and
$$IV=(-1)^k k!{N^{-1}} \sum_{3\mid\text{o}(\psi)\mid k, \text{o}(\psi)>3 }\psi^{-1}(b)\left[\frac {u^k}{k!}\right]
       (1+u^{\text{o}(\psi)/3}+u^{2\text{o}(\psi)/3})^{\frac {N}{\text{o}(\psi)}}.$$
 Putting $I, II=III+IV$ and the above equalities into (\ref{3.2}), and noting that $\sum_{2 \mid k} {2N/3 \choose k}=\sum_{2 \nmid k}{2N/3 \choose k}$, one sees that
\begin{align*}
&M(b)=\sum_{2 \mid k}M(k, b, D)-\sum_{2 \nmid k}M(k, b, D)\\
&=(\sum_{2 \mid k}\frac 1 {k!}III-\sum_{2 \nmid k}\frac 1 {k!}III)+(\sum_{2\mid k} \frac 1 {k!}I
-\sum_{2\nmid k} \frac 1 {k!}I)+(\sum_{2\mid k}\frac 1 {k!} IV
-\sum_{2\nmid k}\frac 1 {k!} IV).
 \end{align*}
One computes
\begin{align*}
|(\sum_{2\mid k}\frac 1 {k!} I
-\sum_{2\nmid k}\frac 1 {k!}I)|&=\frac 1N| \sum_{k=0}^{2N/3}\sum_{3\nmid\text{o}(\psi), 1\ne \order(\psi)\mid k}\psi^{-1}(b)\left[\frac {u^k}{k!}\right]  {(1-u^{\order(\psi)})}^{{\frac {2N}{3\order(\psi)}}}|\\
&\leq \frac{\phi(N)}{N}| \sum_{k=0}^{2N/3}\left[\frac {u^k}{k!}\right]  {(1+u^2)}^{{\frac {N}{3}}}|\\
&=\frac{\phi(N)}{N}\cdot 2^{\frac N3}\\
&\leq 2^{\frac N3},
  \end{align*}
\begin{align*}
(\sum_{2\mid k}\frac 1 {k!} III
-\sum_{2\nmid k}\frac 1 {k!}III)
  &=\frac  2N \sum_{k=0}^{2N/3}\frac 1 {k!}\left[\frac {u^k}{k!}\right]  \frac {(1-u^3)^{N/3}}{(1-u)^{N/3}}\\
&=\frac 2 N \sum_{k=0}^{2N/3}[u^k](1+u+u^2)^{N/3}\\
&=\frac {2\cdot 3^{N/3}} {N},
  \end{align*}
  and
\begin{align*}
&|(\sum_{2\mid k}\frac 1 {k!} IV
-\sum_{2\nmid k}\frac 1 {k!}IV)|\\
 &=\frac  1N |\sum_{k=0}^{2N/3} \sum_{3\mid\text{o}(\psi)\mid k, \text{o}(\psi)>3 }\psi^{-1}(b)\left[\frac {u^k}{k!}\right]
       (1+u^{\text{o}(\psi)/3}+u^{2\text{o}(\psi)/3})^{\frac {N}{\text{o}(\psi)}}|\\
&\leq \frac  {\phi(N)}N \sum_{k=0}^{2N/3} \left[\frac {u^k}{k!}\right]
       (1+u^{\text{o}(\psi)/3}+u^{2\text{o}(\psi)/3})^{\frac {N}{\text{o}(\psi)}}\\
&=\frac  {\phi(N)}N \cdot 3^{N/6}\\
&\leq 3^{N/6}.
   \end{align*}
Finally we obtain
$$M(b)\geq \frac {2\cdot 3^{N/3}} {N}-2^{\frac N3}-3^{\frac N6}.$$
Similarly, one has
$$M(b)\leq \frac {2\cdot 3^{N/3}} {N}+2^{\frac N3}+3^{\frac N6}.$$
Thus we obtain $|M(b)-\frac {2\cdot 3^{N/3}} {N}|\leq 2^{\frac N3}+3^{\frac N6}$.\end{proof}
 \section{Proof of Theorem \ref{thm1.3}}
 \begin{proof}
  By Euler's Pentagon Theorem,
  \begin{align*}
  [x^{pk}]\frac {\prod_{n\geq 1} (1-x^n)}{ \prod_{n\geq 1}(1-x^{pn})}=[x^{pk}]\frac { \sum_{n\in \mathbb{Z}} (-1)^n x^{n(3n-1)/2}}{ \prod_{n\geq 1}(1-x^{pn})},
  \end{align*}
and we then have
    \begin{align*}
    \sum_{k\geq 0} b_{p, pk} x^k&=\frac{\sum_{n=pj} (-1)^j x^{j(3pj-1)/2}+
    \sum_{n=pj+\frac{pt+1}{3}} (-1)^j x^{(3j+t)(pj+\frac{pt+1}{3})/2}} {\prod_{n\geq 1}(1-x^{n})}\\
        &=\frac {z_1+z_2}{\prod_{n \geq 1} (1-x^{n})}.
          \end{align*}
            Here $t$  is the smallest positive integer such that $3 \mid (pt+1)$. By Andrew's formula (\cite{A1}, p. 22, Corollary 2.9)(put $k=\frac {3p-1}{2}, i=\frac {3p+1}{2}$ and $k=\frac {3p-1}{2}, i=\frac {(3-2t)p-1}{2}$ respectively),
  $$z_1=\sum_{n=pj} (-1)^j x^{j(3pj-1)/2}=\prod_{n\geq 0} (1 - x^{3pn+3p})(1 - x^{3pn+\frac {3p-1}{2}})(1 - x^{3pn+\frac {3p+1}{2}}),$$
  and
    \begin{align*}
    z_2&= \sum_{n=pj+\frac{pt+1}{3}} (-1)^j x^{(3j+t)(pj+\frac{pt+1}{3})/2}\\
&=x\prod_{n\geq 0} x^{\frac {t(pt+1)}6}(1 - x^{3pn+3p})(1 - x^{3pn+\frac {(3-2t)p-1}{2}})(1 - x^{3pn+ \frac {(3+2t)p+1}{2}  }).
   \end{align*}
Then we have
   $$b_{p, pk}=P_{\not\equiv 0,  \frac {3p-1}{2}, \frac {3p+1}{2}\pmod{3p}}(k)+P_{\not\equiv 0, \frac {(3-2t)p-1}{2}, \frac {(3+2t)p+1}{2}\pmod{3p}}(k-{\frac {t(pt+1)}6})\geq0.$$
    Note that $t$  is the smallest positive integer such that $3 \mid (pt+1)$, i.e., $t=1$ or $t=2$.
\end{proof}
 {\bf Acknowledgements.} The author wishes to thank Professor Richard Stanley, Professor Wadim  Zudilin and Dr. Mikhail Tyaglov for their helpful discussions. He also wishes to thank the anonymous reviewers 
 for their constructive comments.

\end{document}